\pdfoutput=1
\RequirePackage{ifpdf}
\ifpdf 
\documentclass[pdftex]{sigma}
\else
\documentclass{sigma}
\fi

\numberwithin{equation}{section}

\newtheorem{Theorem}{Theorem}[section]
\newtheorem{Proposition}[Theorem]{Proposition}
 { \theoremstyle{definition}
\newtheorem{Remark}[Theorem]{Remark} }

\begin{document}
\allowdisplaybreaks

\newcommand{\arXivNumber}{2108.02603}

\renewcommand{\thefootnote}{}

\renewcommand{\PaperNumber}{090}

\FirstPageHeading

\ShortArticleName{Spinors, Twistors and Classical Geometry}

\ArticleName{Spinors, Twistors and Classical Geometry\footnote{This paper is a~contribution to the Special Issue on Twistors from Geometry to Physics in honor of Roger Penrose. The~full collection is available at \href{https://www.emis.de/journals/SIGMA/Penrose.html}{https://www.emis.de/journals/SIGMA/Penrose.html}}}

\Author{Nigel J.~HITCHIN}

\AuthorNameForHeading{N.J.~Hitchin}

\Address{Mathematical Institute, Woodstock Road, Oxford, OX2 6GG, UK}
\Email{\href{mailto:hitchin@maths.ox.ac.uk}{hitchin@maths.ox.ac.uk}}
\URLaddress{\url{https://people.maths.ox.ac.uk/hitchin/}}

\ArticleDates{Received August 07, 2021, in final form October 07, 2021; Published online October 10, 2021}

\Abstract{The paper studies explicitly the Hitchin system restricted to the Higgs fields on a fixed very stable rank 2 bundle in genus 2 and 3. The associated families of quadrics relate to both the geometry of Penrose's twistor spaces and several classical results.}


\Classification{14H60; 32L25}

\begin{flushright}
\begin{minipage}{60mm}
\it Dedicated to Roger Penrose\\ on the occasion of his 90th birthday
\end{minipage}
\end{flushright}

\renewcommand{\thefootnote}{\arabic{footnote}}
\setcounter{footnote}{0}

\section{Introduction}

In 1974 I returned to Oxford after three postdoctoral years in the USA and found there a~distinctly new style of mathematics emanating from Roger Penrose's group. From seminars and talking to students and postdocs I gradually learned what twistor theory was all about and it subsequently became a central idea
in much of my work, in four and fewer dimensions and higher, especially in the context of hyperk\"ahler geometry. In fact my papers listed in the bibliography \cite{Hitchin1979,Hitchin1981,Hitchin1982,Hitchin1982b,Hitchin1984,Hitchin1995,Hitchin1998,Hitchin2013, Hitchin2014, Hitchin2021,Hitchin/Karlhede/Lindstrom/Rocek} are all about the twistor interpretation of various constructions in differential geometry.

This article is based on a recent encounter with spinors and twistors in trying to understand in more detail the integrable system associated to moduli spaces of Higgs bundles, part of a~project with Tam\'as Hausel, and by looking at curves of genus $2$ and $3$ we come face to face with some classical results in projective geometry. A common theme seems to be families of quadrics, in particular the Klein quadric which is Penrose's compactified, complexified Minkowski space, and looking at these issues through twistor-coloured glasses reveals some new approaches to these moduli spaces.

To be more specific, we observe a relationship between the intersection of two quadrics in ${\rm P}^5$ viewed as both a twistor space and a moduli space of stable bundles on a curve of genus $2$, and use the formalism of self-duality in four dimensions to describe the integrable system for a curve of genus $3$.
The classical contributors are Pl\"ucker, Salmon, Cayley, Hesse \dots\ but the modern reader is referred to~\cite{Dolgachev} for an excellent account of many of these themes.

\section{Spinors and twistors}
\subsection{Four dimensions}\label{four}
The first thing one learns in twistor theory is how to deal with four dimensions using two-component spinors. If $S_+$, $S_-$ are $2$-dimensional complex vector spaces with skew forms~$\epsilon_+$,~$\epsilon_-$ then the tensor product $V=S_+\otimes S_-$ has an induced non-degenerate symmetric bilinear form. In particular a null vector $v$ can be written $v=\phi\otimes \psi$.

The exterior product now has a decomposition
$\Lambda^2V=S_+^2\oplus S_-^2$ as the direct sum of two $3$-dimensional spaces, the second symmetric power of $S_+$ or $S_-$. If $V$ is the cotangent space~$T^*$ of a manifold these are the self-dual and anti-self-dual $2$-forms. The spaces $S^2_+$, $S^2_-$ are also naturally the $3$-dimensional Lie algebras of trace zero endomorphisms of $S_+$, $S_-$.

The curvature tensor of a metric, as a section of $\Lambda^2T^*\otimes \Lambda^2T^*$, decomposes into the scalar curvature, the trace-free Ricci tensor which lies in $S^2_+\otimes S^2_-$ and the self-dual and anti-self-dual components of the Weyl tensor which lie in $S^4_+$, $S^4_-$ respectively.

\subsection{Linear twistor theory}
The twistor interpretation of complexified Minkowski space begins with ${\rm P}^3$, complex projective $3$-space. The {\it lines} in ${\rm P}^3$ are parametrized by a $4$-dimensional quadric $Q^4\subset{\rm P}^5$ and the conformal geometry on $Q^4$ is defined by asserting that two points are null separated if the corresponding lines in ${\rm P}^3$ intersect.

A {\it point} in the twistor space ${\rm P}^3$ defines the family of lines passing through it which is a null plane in $Q^4$, an $\alpha$-plane. A {\it plane} in ${\rm P}^3$ (equivalently a point in the dual projective space) defines the lines within it which consists also of a null plane but belongs to a different family, the $\beta$-planes.

A line in the {\it quadric} is a {\it null geodesic} and this is given by a pair of a point in ${\rm P}^3$ lying in a plane -- the line is the intersection of the corresponding $\alpha$-plane and $\beta$-plane. Thus the basic constituents of the conformal structure are encoded in purely geometrical terms within a~three-dimensional space.

Underlying this correspondence is the special isomorphism ${\rm Spin}(6,{\mathbb C})\cong {\rm SL}(4,{\mathbb C})$. The first group consists of the conformal transformations of the quadric, and the twistor space and its dual are the projective spaces of the two spin representations $V_+$, $V_-$.

\subsection{Nonlinear twistor theory} \label{graviton}
The nonlinear graviton construction \cite{Penrose} replaces ${\rm P}^3$ by a more general complex $3$-manifold $Z$, which contains a family of rational curves -- the twistor lines~-- copies of ${\rm P}^1$ with normal bundle $N\cong \mathcal{O}(1)\oplus \mathcal{O}(1)$ where $\mathcal{O}(n)$ denotes the line bundle of degree $n$ on ${\rm P}^1$. Kodaira's deformation theorem implies that there is a complete four-dimensional family $M^4$ of such lines and null separation is defined as in the linear case by intersection of the rational curves.

The Weyl tensor is conformally invariant and in this case the component in $S^4_-$ vanishes. The lines through a point in $Z$ define null surfaces like the $\alpha$-planes but the non-vanishing of the self-dual component of the Weyl tensor obstructs the existence of $\beta$-planes.

\section{Integrable systems}
These are the so-called Hitchin systems introduced in~\cite{Hitchin1987}. The setting is as follows in the simplest case: we have a compact Riemann surface, or algebraic curve $C$, of genus $g\ge 2$ and a~stable rank $2$ vector bundle $E$ over $C$. The moduli space ${\mathcal N}$ of such bundles is a projective variety of dimension $3g-3$ and the cotangent space at a point $E$ is $H^0(C,\operatorname{End}_0E\otimes K)$ the space of holomorphic trace zero sections $\Phi$ of the bundle of endomorphisms twisted by $K$, the canonical bundle of holomorphic $1$-forms on~$C$. There is a natural map $\operatorname{tr} \Phi^2$ from $H^0(C,\operatorname{End} _0E\otimes K)$ to~$H^0\big(C,K^2\big)$. Both spaces have dimension $3g-3$ and so we have a $2n$-dimensional mani\-fold~$T^*{\mathcal N}$ with $n$ holomorphic functions, quadratic in the fibres. These Poisson-commute and define a completely integrable system. They are quadratic and so on open sets they describe geodesic flows which are completely integrable.

The simplest example is $g=2$ with $\Lambda^2E$ the trivial bundle, in which case ${\mathcal N}$ is ${\rm P}^3$. The integrable system was calculated in~\cite{Geemen/Previato}, with a beautiful formula in~\cite{Gawedzki/Bich} and a less attractive but more concrete expression in~\cite{Loray/Heu}. Nevertheless, if you take a specific~$E$ and ask what the quadratic map from ${\mathbb C}^3$ to ${\mathbb C}^3$ looks like, it is difficult to realize.

As far as I am aware there is no discussion in the literature of an explicit form for $g=2$ and bundles of {\it odd} degree (and $\Lambda^2E$ fixed) where, thanks to \cite{Narasimhan/Ramanan1969,Newstead} we know that the moduli space~${\mathcal N}$ is the intersection of two quadrics in~${\rm P}^5$.

 We can describe this situation as three quadratic forms on ${\mathbb C}^3$, or in classical geometrical language a {\it net of conics}. We begin next, given our familiarity with the quadrics in twistor theory, to look at this intersection, its relationship with stable bundles, and the associated net.

\section{The intersection of two quadrics}
\subsection{Projective bundles}
Complexified compactified Minkowski space is a quadric $Q_1$ in ${\rm P}^5$. We are now given another generic quadric $Q_2$. This means the zero sets of two quadratic forms $q_1$, $q_2$ on ${\mathbb C}^6$ and indeed {\it a~pencil} of quadrics $z_1q_1+z_2q_2$ for $[z_1,z_2]$ homogeneous coordinates of a point $z\in {\rm P}^1$. The singular quadrics in the pencil occur when $\det (z_1q_1+z_2q_2)=0$ which gives six points \mbox{$a_1,\dots, a_6\in {\rm P}^1$}. A point $x\in Q_1\cap Q_2$ lies in all quadrics $Q_z$ of the pencil.

Each point $z\ne a_i$ in ${\rm P}^1$ defines a nonsingular quadric $Q_z$ and we can take the corresponding space of $\alpha$-planes, a projective space $P_z$. This is not in fact well-defined because there is no absolute way to distinguish $\alpha$- and $\beta$-planes, but if we take a double covering of ${\rm P}^1$ branched over the six points $a_i$ there is a well-defined choice. The curve $C$ so constructed is of genus $2$ and has an involution $\sigma$ exchanging the sheets of the covering.

When $z=a_i$ the quadric is singular but the planes in it (just one family now) are also parametrized by ${\rm P}^3$. One way to think of this is via the group ${\rm Spin}(6,{\mathbb C})$ with its two $4$-dimensional spin representations $V_+$ and $V_-$ -- they become isomorphic as representations of ${\rm Spin}(5,{\mathbb C})$. The result is a 4-dimensional manifold $M^4$ which is a ${\rm P}^3$-fibration over the curve $C$.

A point $x\in Q_1\cap Q_2$ now defines a twistor line in each fibre $P_z$ of $M^4$ and hence a ${\rm P}^1$ bundle over~$C$ -- and this is the projectivized bundle ${\rm P}(E)$ in the description of the moduli space in~\cite{Narasimhan/Ramanan1969, Newstead}! Of course, the main work in those papers is proving that every stable bundle is obtained this way.

Fixing $\Lambda^2E$ means that $E$ and $F$ define the same projective bundle if $F=E\otimes L$ and $L^2$ is trivial. Then the moduli space of projective bundles is ${\mathcal N}/\Gamma$ where $\Gamma = H^1(C,{\mathbb Z}_2)\cong {\mathbb Z}_2^4$, and since tensoring with $L$ has no effect on $\operatorname{End}_0E$, this moduli space is all we need.

Twistors give us an insight into another construction of the moduli space which will be useful in the next section. (This is Atiyah's Smith's Prize essay \cite{Atiyah} when he was a graduate student in Cambridge at the same time as Penrose.) It is known classically that through a general point $x\in Q_1\cap Q_2$ there pass four lines. Choose a pair $x\in \ell\subset Q_1\cap Q_2$, then it gives a line in each quadric of the pencil and hence a point in a plane in each ${\rm P}^3$ fibre~$P_z$. The point defines a~section over $C$ and since $x\in \ell$ the section lies in the ruled surface ${\rm P}(E)$. Equivalently we have a sub-line bundle of~$E$. We can normalize the choice of representative vector bundle by asking that the subbundle is trivial and then $E$ is an extension $\mathcal{O}\rightarrow E\rightarrow L$ where~$L$ has degree~$1$.

The extension class lies in the sheaf cohomology space $H^1(C, L^*)$ which is dual to $H^0(C, LK)$ and by Riemann--Roch these are two-dimensional. The projective space of a 2-dimensional vector space $V$ is canonically isomorphic to its dual (to each subspace of $V$ associate its annihilator in~$V^*$), so the extension class $[\alpha]$ is determined by a section $s$ of $LK$, which is of degree $3$. Then~$s$ vanishes on a divisor $p+q+r$, and its annihilator satisfies $[\alpha] s=0$ in the duality pairing $H^1(C,L^*)\otimes H^0(C,LK)\rightarrow H^1(C,K)={\mathbb C}$.

What Atiyah shows is that the four choices of line define four unordered triples of points in~$C$:
$(p,q,r)$, $(p,\sigma(q),\sigma(r))$, $(\sigma(p),q,\sigma(r))$, $(\sigma(p),\sigma(q),r)$. Then, taking the images of $p$, $q$, $r$ under the projection $\pi\colon C\rightarrow {\rm P}^1$, a model for $Q_1\cap Q_2/\Gamma$ is a double covering of the symmetric product $S^3\big({\rm P}^1\big)\cong {\rm P}^3$ branched over six planes $(a_i,\pi(q),\pi(r))$, $1\le i\le 6$.

\subsection{The net of conics}\label{net}
We shall use Atiyah's approach to calculate $\operatorname{tr}\Phi^2\colon H^0(C,\operatorname{End}_0E\otimes K) \rightarrow H^0\big(C,K^2\big)$ in this genus $2$ case. Via the projection $\pi\colon C\rightarrow {\rm P}^1$ the canonical bundle is isomorphic to $\pi^*\mathcal{O}(1)$. Each section of $K$ therefore has a divisor of the form $p+\sigma(p)$. The $3$-dimensional space of sections of $K^2\cong \pi^*\mathcal{O}(2)$ is given by pulling back $H^0\big({\rm P}^1,\mathcal{O}(2)\big)$.

To work with the extension we adopt a Dolbeault approach and take a $C^{\infty}$ splitting of $E$. Then
 the holomorphic structure is defined by the new $\bar\partial$-operator $(s_1,s_2)\mapsto \big(\bar\partial s_1+\alpha s_2, \bar\partial s_2\big)$
for $\alpha\in \Omega^{0,1}(C,L^*)$ representing a class $[\alpha]$ in $H^1(C,L^*)$. As we have seen this is determined by a~divisor $D=p+q+r$ of a~section~$s$ such that~$[\alpha]$ annihilates~$s$, or equivalently $\alpha s=\bar\partial t$.

A section
\[\Phi=\begin{pmatrix}a & b\\
c & -a\end{pmatrix}\] of $\operatorname{End}_0E\otimes K$ is now holomorphic if
\begin{equation}
\bar\partial c=0,\qquad \bar\partial a+\alpha c=0,\qquad \bar\partial b - 2\alpha a=0.
\label{holo}
\end{equation}

Multiplication by $s$, a section of $KL$, gives an exact sequence of sheaves
\begin{displaymath}0\rightarrow \mathcal{O}_C(L^*)\stackrel{s}\rightarrow \mathcal{O}_C(K)\rightarrow \mathcal{O}_D(K)\rightarrow 0\end{displaymath}
and in the exact cohomology sequence to say that $[\alpha]$ annihilates $s$ means it lies in the image of $H^0(D, K)\rightarrow H^1(C,L^*)$. This implies that
 $\alpha=\bar\partial u/s$ where $u$ is a $C^{\infty}$ section of $K$ holomorphic in a neighbourhood of $D$ and vanishing outside a slightly larger one. Here $u$ is well-defined modulo the restriction of a section of $K$ to $D$.

A bundle $E$ is called {\it very stable} if there is no nilpotent section of $\operatorname{End} E\otimes K$.

\begin{Proposition} Let $E$ be a very stable rank $2$ bundle of odd degree on a curve~$C$ of genus~$2$. Then the net of conics defined by $\operatorname{tr} \Phi^2$ is spanned by quadratic forms of the form $x^2$, $y^2$, $z^2$, geometrically three double lines.
\end{Proposition}

\begin{proof}
We use the description of the bundle as an extension. If the degree one line bundle $L^*K$ has a section then there is a clear nilpotent $\Phi$, taking $a=0=c$ in~(\ref{holo}), so under the assumption that $E$ is very stable, $H^0(C,L^*K)=0$. Now $K(-q-r)\cong L^*(p)$ and if this has a section then $L\cong \mathcal{O}_C(p)$, but then a holomorphic form vanishing at $p$ gives a section of $L^*K\cong K(-p)$. It follows that $K(-q-r)$ has no sections and by adding one-forms that are non-zero at $q$ and $r$ we can choose $u$ to vanish at~$q$,~$r$.

First consider the case $c=0$ then $a$ in (\ref{holo}) is a holomorphic $1$-form $h$. By Riemann--Roch $H^0(C,L^*K)=0$ implies $H^1(C,L^*K)$ =0, therefore $[\alpha] a\in H^1(C,L^*K)$ vanishes which means we can solve~(\ref{holo}) uniquely for $b$. This gives two sections $\Phi_1$, $\Phi_2$ of $\operatorname{End}_0E\otimes K$ taking $h=h_1,h_2$ spanning~$H^0(C,K)$.

To determine $b$ explicitly, the equation $\bar\partial b - 2\alpha a=0$ gives $sb-2uh_i$ as a holomorphic section of $K^2$ which, by the choice of~$u$, vanishes at $q$ and $r$. Take $h_1$ to vanish at $q$ and $h_2$ to vanish at $r$ then this is a multiple $\lambda_1$ of $h_1h_2$. So $b_i=(2uh_i+\lambda_ih_1h_2)/s$ and regularity of $b$ at $p$ gives $\lambda_1h_2(p)+2u(p)=0, \lambda_2h_1(p)+2u(p)=0$.

To find another independent section
take $c$, a holomorphic section of $LK$, to be $s$ itself.
Since $\alpha=\bar\partial u/s$ we can now take $a= -u$ to solve $\bar\partial a+\alpha c=0$.
We require
$\bar\partial b=2(-u)\bar\partial u/s=\bar\partial\big({-}u^2\big)/s$.
Now $sb+u^2$, a holomorphic section of~$K^2$, vanishes at $q$, $r$ so $sb =-u^2+\lambda_3h_1h_2$. Here we have $\lambda_3h_1h_2(p)=u^2(p)$ and hence, since $\lambda_1h_2(p)+2u(p)=0$ and $\lambda_2h_1(p)+2u(p)=0$,
it follows that $ \lambda_1\lambda_2=4u^2/h_1h_2(p)=4\lambda_3$
and we have $\Phi_3$.

Since $K=\pi^*\mathcal{O}(1)$, holomorphic 1-forms are pulled back from linear forms $az+b$ so take $h_1=1$, $h_2=z$. Then
for $\Phi=x_1\Phi_1+x_2\Phi_2+x_3\Phi_3$ we get
\[
-\operatorname{tr}\Phi^2=(x_1+zx_2)^2+x_3z( \lambda_1x_1+\lambda_2x_2 +\lambda_1\lambda_2x_3/4).
\] Taking coefficients of $1$, $z$, $z^2$, the $3$-dimensional space of conics is now spanned by
$x_1^2$, $x_2^2$ and $2x_1x_2+x_3(\lambda_1x_1+\lambda_2x_2+\lambda_1\lambda_2x_3)/4$.
Completing the square in the last term gives
\[
\frac{\lambda_1\lambda_2}{4}\left(x_3+\frac{2}{\lambda_1\lambda_2}(\lambda_1x_1+\lambda_2x_2)\right)^2-\frac{\lambda_1}{\lambda_2}x_1^2-\frac{\lambda_2}{\lambda_1} x_2^2,
\]
so there is a basis in which the net is the three-parameter family of quadratic forms spanned by~$x^2$, $y^2$, $z^2$.
\end{proof}

\begin{Remark} What is notable here is that the structure is independent of the curve $C$ and also of the bundle so long as it is very stable. The above method
can be used for other types of stable bundle $E$, for example a generic situation where there is a nilpotent $\Phi$ gives the net spanned by~$x^2$,~$y^2$,~$yz$. In Atiyah's description the bundle ${\rm P}(E)$ for this case is defined by a~generic point on the diagonal $(a,b,b)$ in $S^3{\rm P}^1$.
\end{Remark}

Before considering genus $3$, there is another aspect to the intersection of quadrics which deserves to be mentioned.

\subsection{A nonlinear twistor space}
One outcome of the paper \cite{Hitchin1981} was the observation that the {\it conics} on the intersection of two quadrics have normal bundle $\mathcal{O}(1)\oplus \mathcal{O}(1)$ and hence, as in Section~\ref{graviton}, form a natural example of a nonlinear twistor space. I gave the problem to Jacques Hurtubise, my student at the time, and in \cite{Hurtubise} he produced concrete expressions for the conformal structure and Weyl tensor. Here we consider a more geometrical (but less explicit) viewpoint, making contact with the discussion above.

Firstly, a conic in a projective space is the intersection of a plane with a quadric, and the conics in $Q_1\cap Q_2$ are defined by taking an $\alpha$-plane in one of the quadrics of the pencil and intersecting it with a generic one. This means taking $z\in {\rm P}^1$ and an $\alpha$-plane in $Q_z$, but this is simply a point in $M^4$, our ${\rm P}^3$-bundle over $C$. We conclude that $M^4$ is the corresponding complex space-time with a conformal structure. Actually we must remove from $M^4$ the $\alpha$-planes which intersect $Q_1\cap Q_2$ in a pair of lines~-- this is a quartic surface in each ${\rm P}^3$ fibre, the degree $4$ property giving the four lines through a point as noted above.

A point $x\in Q_1\cap Q_2$, according to the nonlinear graviton construction, must define a null surface in~$M^4$. We have already seen how we obtain a ruled surface from $x$ -- the projective bundle of a stable vector bundle $E$. This is in fact the null surface: a point over $z\in {\rm P}^1$ is an $\alpha$-plane in $Q_z$ through~$x$, and hence all the conics pass through $x$ and are null separated. Then the conformal structure on $M^4$ is defined by the stable bundles on~$C$!

\section{Genus 3 curves}
\subsection{The Coble quartic}

Recall that if $C$ has genus $3$ and is non-hyperelliptic it is a quartic curve in ${\rm P}^2$ and then the canonical bundle $K$ is the restriction of $\mathcal{O}(1)$ so that the 3-dimensional space of holomorphic 1-forms can be identified with linear forms on ${\mathbb C}^3$, and likewise the 6-dimensional space of sections of $K^2$ with quadratic forms on ${\mathbb C}^3$.

The moduli space ${\mathcal N}$ of (semi)stable rank two bundles with $\Lambda^2E$ trivial on a quartic curve $C$ is again known explicitly \cite{Narasimhan/Ramanan1987} as a special singular quartic hypersurface in ${\rm P}^7$, initially studied by A.~Coble in the context of abelian varieties. The decomposable bundles $E=L\oplus L^*$ describe a Kummer variety which is the singular locus. This property together with the invariance under ${\mathbb Z}_2^6\cong H^1(C,{\mathbb Z}_2)$ characterize it. Another property, described by Christian Pauly in~\cite{Pauly}, is that, like its cousin the Kummer quartic surface, its dual is another copy.

Pauly expresses the self-duality in the following way in terms of bundles: if $E$ is a very stable rank $2$ bundle with $\Lambda^2E\cong \mathcal{O}$, then there is a unique stable bundle $F$ with $\Lambda^2F\cong K$ such that $\dim H^0(C, E\otimes F)=4$. He now associates to this a net of quadrics in $\mathrm{P}^3$.

\subsection{Families of quadrics}\label{sd}
Let $E$ be a very stable bundle with corresponding bundle $F$ as above and denote by $V$ the 4-dimensional space $H^0(C, E\otimes F)$. The bundle $E$ has a holomorphic skew-symmetric form $\epsilon$ and $F$ a skew form $\epsilon'$ with values in $K$. Hence, as in Section~\ref{four}, we have a quadratic form on~$V$ with values in the 3-dimensional space $H^0(C,K)$. This is Pauly's net of quadrics. But we can go further, following the spinor description of four dimensional geometry.

There is a natural homomorphism
\[ \Lambda^2V\rightarrow H^0\big(C, \Lambda^2(E\otimes F)\big)\rightarrow H^0\big(C, S^2E\otimes K\big)\]
by contracting with $\epsilon'$ and symmetrizing $E$, or taking the self-dual component in the language of Section~\ref{four}. But $S^2E\cong \operatorname{End}_0 E$ and so we have a map from the $6$-dimensional space $\Lambda^2V$ to the $3g-3=6$-dimensional space $H^0(C, \operatorname{End}_0 E\otimes K)$. Leaving aside for the moment the question of whether this is an isomorphism, it offers a way of evaluating $\operatorname{tr} \Phi^2$ on the image.

On a 4-manifold, a metric on the tangent space $T$ induces one on the self-dual $2$-forms, and identifying $S^2_+$ with $ \operatorname{End} _0 S_+$ this can be taken as $-\operatorname{tr} a^2$. The inner product on the self-dual component $\alpha_+$, $\beta_+$ of $\alpha,\beta\in \Lambda^2T^*$ is given by
\[
(\alpha_+,\beta_+)\omega=\frac{1}{2}\left[(\alpha,\beta)\omega+\alpha\wedge\beta\right],
\]
where $\omega$ is the volume form.
In our situation we have an inner product on $V$ with values in sections of $K$, so choose a basis $v_1,\dots, v_4$ and take $v_1\wedge v_2\wedge v_3\wedge v_4$ to trivialize $\Lambda^4V$ and we have $((v_i\wedge v_j)_+,(v_i\wedge v_k)_+)=(v_i\wedge v_j,v_i\wedge v_k)/2$ but
\[
((v_1\wedge v_2)_+,(v_3\wedge v_4)_+)=\frac{1}{2}\left[(v_1\wedge v_2,v_3\wedge v_4)+\sqrt{\det(v_i,v_j)} \right].
\]
This expression looks odd -- we know that it must take values in $H^0\big(C,K^2\big)$, quadratic functions on ${\mathbb C}^3$, but the last term involves the square root of a quartic.

The only resolution is that $\det(v_i,v_j)= p^2$ {\it modulo the quartic equation of $C$}. We therefore have two quartic curves in~$\mathrm{P}^2$, the original one with equation $\det(v_i,v_j)- p^2=0$ and another one $X$, depending on $E$ with equation $\det(v_i,v_j)=0$. The fact that the difference of the quartic expressions is~$p^2$, a square, means that $X$ and $C$ meet tangentially at $8$ points instead of~$16$. Pauly shows in fact that the space of tangential quartics to~$C$ essentially parametrizes the projective bundles arising from the Coble quartic.

\begin{Remark}
The eight points of intersection determine a pencil of quartics $\det(v_i,v_j)- zp^2=0$, $z\in {\rm P}^1$ and do not pick out $X$ to define the bundle $E$ on $C$. In fact~\cite{Pauly}, $E$ has 8 maximal subbundles $L^*_i$ of degree $-1$ and 28 sections $s_{ij}$ of $L_iL_j$, $i\ne j$, which vanish at the points when the pair coincide. If the divisor of $s_{ij}$ is $x+y$ then the line in ${\rm P}^2$ joining $x$ and $y$ in~$C$ intersects~$X$ as a bitangent. This fixes~$X$ in the pencil.
\end{Remark}

\begin{Remark}
 Choosing a square root $K^{1/2}$ of the line bundle $K$, the bundle $E'=F\otimes K^{-1/2}$ has $\Lambda^2E'$ trivial since $\Lambda^2F\cong K$. The projective bundle ${\rm P}(F)$ is independent of the choice of square root and so the self-duality can be regarded as a transformation on the moduli space of projective bundles. In the spinor description above $E$, $F$ play the role of self-dual or anti-self-dual spin spaces and the inner product changes sign on the $\sqrt{\det(v_i,v_j)}$ factor. Hence we have the involution
$z\mapsto -z$ on the pencil.
\end{Remark}

\subsection{The discriminant}
If $q_i$, $1\le i\le m$ are symmetric $n\times n$ matrices and $Q=z_1q_1+\dots + z_mq_m$ then the {\it discriminant} is the hypersurface in ${\rm P}^{m-1}$ defined by $\det Q=0$. In our context of a quadratic map $H^0(C,\operatorname{End}_0E\otimes K)\rightarrow H^0\big(C,K^2\big)$ it is the space of linear functions on $H^0\big(C,K^2\big)$ which have a non-zero critical point. By Serre duality, the discriminant lies naturally in ${\rm P}\big(H^1(C,K^*)\big)$.

In the case of a net of conics, with $m=n=3$, the discriminant is a plane cubic curve, but as we saw in Section~\ref{net} the case for $\operatorname{tr} \Phi^2$ in genus~$2$ gave three lines, a non-generic situation. However, before considering a particular case of the genus~$3$ discriminant
 we return to an issue from the last section:

\begin{Proposition} The map $\Lambda^2V\rightarrow H^0\big(C, S^2E\otimes K\big)$ is an isomorphism.
\end{Proposition}
\begin{proof} The sections of $K$ given by $(v_i,v_j)$ and $p$ define the quadratic form on $\Lambda^2V$ relative to a basis. Suppose $\alpha\in \Lambda^2V$ maps to zero
in $H^0\big(C, S^2E\otimes K\big)$, then since the form is $\operatorname{tr} \Phi^2$ on the image, $\alpha$ is orthogonal to all vectors and so the discriminant vanishes identically.

We described the form as $Q_2+pQ_0$ where $Q_0$ is the exterior product $Q_0(\alpha,\beta)=\alpha\wedge\beta$ and~$Q_2$ is the inner product on~$\Lambda^2V$ induced by~$Q$ on~$V$. Then a calculation gives
\begin{equation*}
\det (Q_2+pQ_0)=\big(\det Q-p^2\big)^3.
\end{equation*}
We should be careful in interpreting this formula. The discriminant hypersurface lies in $\mathrm{P}^5={\rm P}\big(H^1(C,K^{*})\big)$ and is defined by a polynomial in $H^0\big(C,K^2\big)$, without imposing any relations like $\big(x^2\big)\big(y^2\big)=(xy)^2$. These relations define a Veronese surface, an embedding of ${\rm P}^2$ in ${\rm P}^5$ and $\det Q-p^2$ is the equation of $C\subset {\rm P}^2$, so this represents the intersection of the discriminant with~${\rm P}^2$. However, since it is not identically zero, the map in the proposition is injective and since both spaces have dimension~$6$, it is an isomorphism.
\end{proof}

 \subsection{Syzygetic tetrads}
 Consider finally the special case where the quartic $X$ defining the vector bundle $E$ consists of four lines. Then $C$ has equation $\ell_1\ell_2\ell_3\ell_4=p^2$ for linear forms $\ell_i$. Since $X$ meets $C$ tangentially, these lines are four bitangents, a so-called {\it syzygetic tetrad}~-- the 8 points of intersection lie on a conic $p=0$.
 It is known classically that any quartic can be so expressed and in 315 different ways~\cite{Dolgachev}. This only accounts for a finite number of bundles~$E$, but it is a case which we can work out explicitly and see that, unlike the case of genus~2, there are continuous parameters in the equivalence class of the family of quadrics.

 One should note that bitangents are related to spinors on~$C$~-- in two dimensions the spin bundle is a square root $K^{1/2}$ of the canonical bundle and a holomorphic section is a solution to the Dirac equation. There are $\big\vert H^1(C,{\mathbb Z}_2)\big\vert=2^{2g}=2^6=64$ of these and $28$ have a~section~$s$. Then $s^2$ is a section of~$K$ with two double zeros~-- the bitangent. The notion of bitangent is thus intrinsic for~$C$, as is the syzygetic tetrad, for this is given by four square roots whose tensor product is~$K^2$.

 \begin{Proposition}
Let $E$ be the rank $2$ bundle on a quartic curve $C$ defined by a syzygetic tetrad. Then the family of quadrics is spanned by equations of the form
\[
x_i^2+2a_i(x_1y_1+x_2y_2+x_3y_3), \qquad y_i^2+2b_i(x_1y_1+x_2y_2+x_3y_3).
\]
for $i=1,2,3$. The six coefficients $a_i$, $b_i$ define the equation of $C$.
\end{Proposition}
\begin{proof}
We normalize the equations of the four lines and then the quadratic form $Q$ on $V=H^0(C,E\otimes F)$ has diagonal entries $x$, $y$, $z$, $x+y+z$.
 Write a general quadratic polynomial in $(x,y,z)$ as
 \[ v_1yz+v_2zx+v_3xy+(u_1x+u_2y+u_3z)(x+y+z).\]
Then, using the inner products of self-dual components as in Section~\ref{sd}, the $6\times 6$ symmetric matrix of linear forms, with $p=\sum_{i=1}^3 a_iu_i+b_iv_i$, can be written
 \[ \begin{pmatrix} u_1 & 0 & 0 & p & 0 & 0\\
 0 & u_2 & 0 & 0 & p & 0\\
 0 & 0 & u_3 & 0 & 0 & p\\
 p & 0 & 0 & v_1 & 0 & 0 \\
 0 & p & 0 & 0 & v_2 & 0\\
 0 & 0 & p & 0 & 0 & v_3\end{pmatrix}
 \]
so the family of quadrics is spanned, for $i=1,2,3$, by
\begin{equation}
x_i^2+2a_i(x_1y_1+x_2y_2+x_3y_3), \qquad y_i^2+2b_i(x_1y_1+x_2y_2+x_3y_3).
\label{algebra}
\end{equation}
\begin{Remark} One may regard these expressions as the relations in a commutative algebra. In the more general context of Higgs bundles such algebras include the cohomology of Grassmannians and products thereof~\cite{Hausel/Hitchin}. So for genus two the relations $x^2=y^2=z^2=0$ generate the cohomology of ${\rm P}^1\times {\rm P}^1\times {\rm P}^1$ and the above algebra can be viewed as a deformation of $H^*\big(\big({\rm P}^1\big)^6,{\mathbb C}\big)$.
\end{Remark}
 It is clear from the matrix that the discriminant is the degree $6$ hypersurface which consists of the union of three singular quadrics $u_iv_i=p^2$, $i=1,2,3$ and we are concerned with equivalence up to projective transformations -- can the relations in (\ref{algebra}) be reduced to standard form as in the genus $2$ situation? Although what is at issue now are three separate quadrics $q_i$ rather than the net they generate we can still consider the discriminant
\[ \det\left(\sum_{i=1}^3 z_i q_i\right)=z_1z_2z_3(z_1z_2z_3-2(a_1b_1z_2z_3+a_2b_2z_3z_1+a_3b_3 z_1z_2)(z_1+z_2+z_3))\]
 and we observe three invariants $a_ib_i$, which are essential parameters for the quadratic map. The six coefficients $a_i$, $b_i$ correspond to the $3g-3=6$ moduli of the curve $C$, so for this case different quartic curves can define isomorphic algebras with relations of the form~(\ref{algebra}).
 \end{proof}

 The decomposition of the discriminant above into three components reflects the fact that in this case the bundle~$\operatorname{End}_0E$ is a direct sum $U_1\oplus U_2\oplus U_3$ where $U_i^2$ and $U_1U_2 U_3$ are trivial. This relates to the syzygetic tetrad as follows: each element of the tetrad gives a square root of $K$: $K^{1/2}$, $K^{1/2}U_1$, $K^{1/2}U_2$, $K^{1/2}U_3$ and since the product of these is~$K^2$ we have $U_1U_2U_3$ trivial.

 Writing the last relation additively (considering the divisor class of the $U_i$) we obtain $U_1+U_2+U_3=0$ which gives a 2-dimensional vector subspace of $H^1(C,{\mathbb Z}_2)$. It is in fact isotropic with respect to the intersection form. This is where
 the number 315 arises~-- counting isotropic planes gives $\big(2^6-1\big)\big(2^5-2\big)/\vert {\rm GL}(2,{\mathbb Z}_2)\vert=315$.

\subsection*{Acknowledgements}
The author thanks Tam\'as Hausel for helpful comments.

\pdfbookmark[1]{References}{ref}
\LastPageEnding

\end{document}